\documentclass[leqno]{amsart}
\title[Translation invariant symmetric polynomials]{On translation
  invariant symmetric polynomials and Haldane's conjecture}
\author{Jesse Liptrap}
\address{Mathematics Department\\
  University of California\\
  Santa Barbara, CA 93106\\
  U.S.A.}
\email{jliptrap@math.ucsb.edu}

\usepackage{graphicx,url,multirow}
\input{xy}
\xyoption{all}
\newtheorem*{mainthm*}{Main Theorem}
\newtheorem{thm}{Theorem}[section]
\newtheorem*{thm*}{Theorem}
\newtheorem{cor}[thm]{Corollary}
\newtheorem*{cor*}{Corollary}
\newtheorem{lem}[thm]{Lemma}
\newtheorem*{lem*}{Lemma}
\newtheorem{obs}[thm]{Observation}
\newtheorem*{obs*}{observation}
\newtheorem{prop}[thm]{Proposition}
\newtheorem*{prop*}{Proposition}
\newtheorem{conj}[thm]{Conjecture}
\theoremstyle{remark}

\newtheorem*{rem*}{Remark}
\theoremstyle{definition}
\newtheorem{defn}[thm]{Definition}
\newtheorem*{defn*}{Definition}

\newcommand{\comment}[1]{}

\newcommand{\F}{\ensuremath{\mathbb{F}}}
\newcommand{\C}{\ensuremath{\mathbb{C}}}
\newcommand{\zz}{\ensuremath{z_1, \ldots, z_n}}
\newcommand{\avg}{\ensuremath{\mathrm{avg}}}
\newcommand{\trans}[1]{\ensuremath{T}}
\newcommand{\sym}[1]{\ensuremath{S}}
\newcommand{\symgrp}[1]{\ensuremath{\mathrm{Sym}(#1)}}
\newcommand{\lex}{\mathrm{lex}}

\begin{document}

\begin{abstract}
  We show that the ring of translation invariant symmetric polynomials in $n$
  variables is isomorphic to the full polynomial ring in $n-1$ variables, in
  characteristic $0$.  We disprove a conjecture of Haldane regarding the
  structure of such polynomials.  Our motivation is the fractional quantum Hall
  effect, where translation invariant (anti)symmetric complex $n$-variate
  polynomials characterize $n$-electron wavefunctions.
\end{abstract}
\maketitle

\section{Introduction}
A polynomial $p(\zz)$ is \emph{translation invariant} if
\[p(z_1+c, \ldots, z_n+c)=p(z_1, \ldots, z_n)\] for all $c$.  Over $\C$, such a
polynomial might yield a quantum mechanical description of $n$ particles in the
plane.  Such is the nature of fractional quantum Hall states, where in addition
our polynomials are symmetric or antisymmetric (Wen and Wang \cite{wenwang}).
Thus we are led to the study of translation invariant (anti)symmetric
polynomials.  Antisymmetric polynomials fortunately do not need special
treatment, as they are merely symmetric polynomials multiplied by the
Vandermonde determinant $\prod_{i<j}(z_i-z_j)$.

This elementary note contains two main results on the structure of translation
invariant symmetric polynomials.  First, a simple description of the ring of
all such polynomials (Corollary~\ref{isom}).  Second, a counterexample to
Haldane's conjecture~\cite{haldane} that every homogeneous translation
invariant symmetric polynomial satisfies a certain physically convenient
property (Proposition~\ref{counterexample}).  More precisely, each symmetric
polynomial $p$ is associated with a finite poset $B(p)$; Haldane conjectured
that if $p$ is homogeneous and translation invariant, then $B(p)$ has a
maximum.  We prove the conjecture for polynomials of at most three variables,
construct a minimal counterexample, and discuss whether a weakened version of
the conjecture holds.

\section{The ring of translation invariant symmetric polynomials}

Throughout, let $\F$ be a field of characteristic $0$.  We begin with just
translation invariance.  Let $\trans{Z} \subseteq \F[\zz]$ be the ring of
translation invariant polynomials.  Imagining $z_1, \ldots, z_n$ as the
locations of $n$ identical particles, and $x_1, \ldots, x_n$ as the
corresponding center of mass coordinates, our main theorem says that $T$
written in terms of $x_1, \ldots, x_n$ is $\F[x_1, \ldots, x_n]$ modulo one
degree of freedom.

\begin{thm}\label{ker_avg}
  Let $\rho \colon \F[x_1, \ldots, x_n] \to \trans{Z}$ be the surjective
  algebra homomorphism
  \[x_i \mapsto z_i-z_\avg\] where $z_\avg = \frac{1}{n}(z_1+\cdots+z_n)$.
  Then $\ker \rho = (x_\avg)$.
\end{thm}

\begin{proof}
  Section~\ref{bigproof}.
\end{proof}

Now let $R \subseteq T$ be the ring of translation invariant symmetric
polynomials in $z_1, \ldots, z_n$, and $S \subseteq \F[x_1, \ldots, x_n]$ be
the ring of symmetric polynomials in $x_1, \ldots, x_n$.

\begin{cor}\label{surj}
  Let $\sigma \colon S \to R$ agree with $\rho$.  Then $\sigma$ is a surjective
  algebra homomorphism, with kernel $(x_1 + \cdots + x_n)$.
\end{cor}

\begin{proof}
  It suffices to show $\rho(\sym{X})=R$.  Clearly $\rho(\sym{X}) \subseteq R$.
  Given $p(z_1, \ldots, z_n) \in R$, translation invariance yields
  \[ p(z_1, \ldots, z_n) = p(z_1-z_\avg, \ldots, z_n-z_\avg) = \rho(p(x_1,
  \ldots, x_n)). \] Thus $\rho(\sym{X})=R$.
\end{proof}

Since $\F$ has characteristic $0$, any element of $S$ can be written uniquely
as a polynomial in the power sum symmetric polynomials $x_1^k + \cdots +
x_n^k$, where $1 \leq k \leq n$.  In other words, the algebra homomorphism
$\theta \colon \F[w_1, \ldots, w_n] \to S$ defined by $\theta(w_k) = x_1^k
+\cdots+ x_n^k$ is an isomorphism.  Note that we could define a different
isomorphism $\theta$ using elementary symmetric polynomials or complete
homogeneous symmetric polynomials.  In any case, $\sigma \theta \colon \F[w_1,
\ldots, w_n] \to R$ is a surjective algebra homomorphism, with kernel $(w_1)$.

\begin{cor}\label{isom}
  The algebra homomorphism $\F[w_2, \ldots, w_n] \to R$ given by
  \[w_k \mapsto (z_1-z_\avg)^k + \cdots + (z_n-z_\avg)^k \] is an isomorphism.
\end{cor}

Next we consider the vector space $R^d$ of all polynomials in $R$ which are
homogeneous of degree $d$. Let $f$ be the above isomorphism.  Since $f(w_k)$ is
homogeneous of degree $k$, we obtain a basis for $R^d$, namely all
\begin{equation}\label{CRRbasis}
  w_{\lambda} = \prod_{k=2}^n f(w_k)^{\lambda_k}
\end{equation}
where $\lambda$ is any partition of $d$ into integers between $2$ and $n$, and
$\lambda_k$ is the multiplicity of $k$ in $\lambda$.  Simon, Rezayi, and Cooper
\cite{simon} prove directly that these $w_{\lambda}$ form a basis of $R^d$,
whereas we have deduced this fact from the ring structure of $R$. Although
\cite{simon} defines $w_\lambda$ using elementary symmetric polynomials rather
than power sum symmetric polynomials, this difference is purely cosmetic.
Since the dimension $m_d$ of $R^d$ is the number of partitions of $d$ into
integers between $2$ and $n$, a generating function for $m_d$ is
\[ \sum_{d=0}^\infty m_dt^d = \prod_{s=2}^n \frac{1}{1-t^s} \]

Finally, we describe the vector space $A \subset \F[z_1, \ldots, z_n]$ of
translation invariant antisymmetric polynomials.  It is well-known that any
antisymmetric polynomial can be written uniquely as $q \Delta$, where $q$ is a
symmetric polynomial and $\Delta$ is the Vandermonde determinant
$\prod_{i<j}(z_i-z_j)$.  Since $\Delta$ is translation invariant, we have $A =
R \Delta$, defining a vector space isomorphism $R \to A$, which sends each
basis \eqref{CRRbasis} to a basis for the vector space of homogeneous
translation invariant antisymmetric polynomials of degree $d + n(n-1)/2$.

\section{Haldane's conjecture}
Every symmetric polynomial is a unique linear combination of symmetrized
monomials, which physicists like to call \emph{boson occupation states}.  We
identify symmetrized monomials with multisets of natural numbers:
\[ [l_1, \ldots, l_n] = \sum_{\sigma \in \symgrp{n}} z_{\sigma(1)}^{l_1} \cdots
z_{\sigma(n)}^{l_n}\] For instance, the multiset $[5,0,0]$ corresponds to the
symmetrized monomial $2z_1^5+2z_2^5+2z_3^5$.  \emph{Squeezing} a symmetrized
monomial $[l_1, \ldots, l_n]$ means decrementing $l_i$ and incrementing $l_j$
for any pair of indices $i,j$ such that $l_i>l_j+1$.  The \emph{squeezing
  order} is a partial order on symmetrized monomials: put $s > t$ iff $t$ can
be obtained from $s$ by repeated squeezing.  For a symmetric polynomial $p$,
let $B(p)$ be the set of all symmetrized monomials with nonzero coefficient in
$p$.  We view $B(p)$ as a poset under the squeezing order and refer to it as
the $\emph{squeezing poset}$ of $p$.
\begin{defn}
  A symmetric polynomial is \emph{Haldane} if its squeezing poset has a
  maximum.
\end{defn}
\begin{conj}[Haldane~\cite{haldane}]
  Every homogeneous translation invariant symmetric polynomial is Haldane.
\end{conj}

\begin{rem*}
  Since squeezing preserves homogeneous degree, Haldane polynomials are
  homogeneous.  Many homogeneous symmetric polynomials are not Haldane, such as
  $[3,3,0]+[4,1,1]$, but these might not be translation invariant.
\end{rem*}

\begin{prop}
  Haldane's conjecture holds for polynomials of $\leq 3$ variables.
\end{prop}

\begin{proof}
  The conjecture is vacuously true for univariate polynomials.  Every bivariate
  symmetrized monomial of homogeneous degree $d$ has the form $[a,b]$, with $a
  + b = d$.  These are linearly ordered under squeezing, so Haldane's
  conjecture is automatic in the bivariate case.

  For the trivariate case, define $\tau \colon \F[z_1, z_2, z_3] \to \F[z_1,
  z_2, z_3,t]$ by
  \[ \tau(p)(z_1, z_2, z_3, t) = p(z_1+t, z_2+t, z_3+t), \] so that $p$ is
  translation invariant iff $\tau(p) = p$.  Define linear endomorphisms
  $\tau_i$ of $\F[z_1, z_2, z_3]$ by $\tau(p) = \sum_{i = 0}^d \tau_i(p) t^i$,
  so that $p$ is translation invariant iff $\tau_i(p)=0$ for all $i>0$.  Then
  \[
  \tau_1([a,b,c]) = a[a-1,b,c] + b[a,b-1,c] + c[a,b,c-1]
  \]
  for all $a,b,c > 0$.  Now suppose $[a,b,c]$ is a maximal element of the
  squeezing poset of some $p \in R_3^d$, with $a \geq b \geq c >0$.  Then
  $[a+1,b,c-1]$ and $[a,b+1,c-1]$ are not in $B(p)$.  The above equation then
  implies that the coefficient of $[a,b,c]$ in $p$ equals $c$ times the
  coefficient of $[a,b,c-1]$ in $\tau_1(p)$.  Thus $\tau_1(p) \neq 0$,
  contradicting the translation invariance of $p$.  Therefore every maximal
  element of $B(p)$ has the form $[a,b,0]$, with $a+b=d$.  These are linearly
  ordered under squeezing; their maximum is the maximum of $B(p)$.
\end{proof}

Any two symmetrized monomials written as weakly decreasing sequences of natural
numbers can be compared lexicographically.  The lexicographic order $>_\lex$ on
symmetrized monomials linearizes the squeezing order.  Let $R_n^d$ be the
vector space of translation invariant symmetric $n$-variate polynomials of
homogeneous degree $d$, and let $L_n^d$ be the set of lexicographic maxima of
squeezing posets of polynomials in $R_n^d$.  Note $|L_n^d| \leq \dim R_n^d$.

\begin{defn}
  A symmetrized monomial $s$ is \emph{completely squeezable} if $s >_\lex t$
  implies $s > t$, for all symmetrized monomials $t$.
\end{defn}

\begin{lem}\label{nosqueeze}
  If every element of $L_n^d$ is completely squeezable, then Haldane's
  conjecture holds for $R_n^d$.  If Haldane's conjecture holds for $R_n^d$,
  then $L_n^d$ is linearly ordered under squeezing.
\end{lem}

\begin{proof}
  The first statement is immediate.  For the second, suppose $m_1,m_2 \in
  L_n^d$ are incomparable.  Let $p_1, p_2 \in R_n^d$ such that $m_i$ is the
  lexicographic maximum of $B(p_i)$ for $i = 1,2$.  W.l.o.g.\ assume $m_1$ is
  lexicographically bigger than $m_2$, and let $c_i$ be the coefficient of
  $m_2$ in $p_i$.  Choose a scalar $c \neq -c_1/c_2$, and let $q = p_1 + cp_2$.
  Then $q \in R_n^d$ and $m_1, m_2 \in B(q)$.  Since $m_1$ is the lexicographic
  maximum of $B(q)$, it is maximal in $B(q)$ under squeezing.  Since $m_1$ and
  $m_2$ are incomparable, $q$ is not Haldane.
\end{proof}

\begin{prop}\label{nonHaldane}
  Haldane's conjecture holds for $R_4^d$ with $d < 14$ but fails for
  $R_4^{14}$.
\end{prop}

\begin{proof}
  It is a straightforward computational linear algebraic exercise to compute
  $L_n^d$ using the basis for $R_n^d$ given by formula~\eqref{CRRbasis}.
  \begin{table}[htbp]
    \centering
    \begin{tabular}{r|l}
      $d$ & $L_4^d$\\
      $0$ & $\emptyset$\\
      $1$ & $\emptyset$\\
      $2$ & $\{[2,0,0,0]\}$\\
      $3$ & $\{[3,0,0,0]\}$\\
      $4$ & $\{[4,0,0,0],[2,2,0,0]\}$\\
      $5$ & $\{[5,0,0,0]\}$\\
      $6$ & $\{[6,0,0,0],[4,2,0,0],[3,3,0,0]\}$\\
      $7$ & $\{[7,0,0,0],[5,2,0,0]\}$\\
      $8$ & $\{[8,0,0,0],[6,2,0,0],[5,3,0,0],[4,4,0,0]\}$\\
      $9$ & $\{[9,0,0,0],[7,2,0,0],[6,3,0,0]\}$\\
      $10$ & $\{[10,0,0,0],[8,2,0,0], [7,3,0,0],[6,4,0,0],[5,5,0,0]\}$\\
      $11$ & $\{[11,0,0,0],[9,2,0,0],[8,3,0,0],[7,4,0,0]\}$\\
      $12$ & $\{[12,0,0,0],[10,2,0,0],[9,3,0,0],[8,4,0,0],[7,5,0,0],[6,6,0,0],[6,4,2,0]\}$\\
      $13$ & $\{[13,0,0,0],[11,2,0,0],[10,3,0,0],[9,4,0,0],[8,5,0,0]\}$\\
      $14$ & $\{[14,0,0,0],[12,2,0,0],[11,3,0,0],[10,4,0,0],[9,5,0,0],[8,6,0,0],[8,4,2,0],$\\
      & $\quad [7,7,0,0]\}$\\
    \end{tabular}
    \caption{Enumeration of lexicographic maxima.} \label{table-lexmaxdata}
  \end{table}
  Since every symmetrized monomial of the form $[a,b,0, \ldots, 0]$ is
  completely squeezable, as is $[6,4,2,0]$, we see that every element of
  $L_4^d$, $d < 14$, is completely squeezable (Table~\ref{table-lexmaxdata}).
  But $L_4^{14}$ is not linearly ordered under squeezing: $[8,4,2,0]$ and
  $[7,7,0,0]$ are incomparable.  Then apply Lemma~\ref{nosqueeze}.
\end{proof}

It is a straightforward computational linear algebraic exercise to construct a
non-Haldane polynomial in $R_4^{14}$ by following the proof of
Lemma~\ref{nosqueeze}.  We get
\begin{equation*}
  \begin{split}
    p =\ &3[8,4,2,0]-3[8,4,1,1]-3[8,3,3,0]+6[8,3,2,1]-3[8,2,2,2]\\
    & +3[7,7,0,0]-42[7,6,1,0]+46[7,5,2,0]+80[7,5,1,1]-22[7,4,3,0]\\
    & -188[7,4,2,1]+112[7,3,3,1]+8[7,3,2,2]+77[6,6,2,0]+70[6,6,1,1]\\
    & -182[6,5,3,0]-700[6,5,2,1]+112[6,4,4,0]+168[6,4,3,1]+1078[6,4,2,2]\\
    &-728[6,3,3,2]+5[5,5,4,0]+1072[5,5,3,1]+246[5,5,2,2]-722[5,4,4,1]\\
    &-2976[5,4,3,2]+1808[5,3,3,3]+1805[4,4,4,2]-1130[4,4,3,3].
  \end{split}
\end{equation*}

\begin{prop}\label{counterexample}
  The polynomial $p$ is a minimal counterexample to Haldane's conjecture.
\end{prop}

\begin{proof}
  \begin{figure}[hbt]
    \centering
    \includegraphics[width=\textwidth]{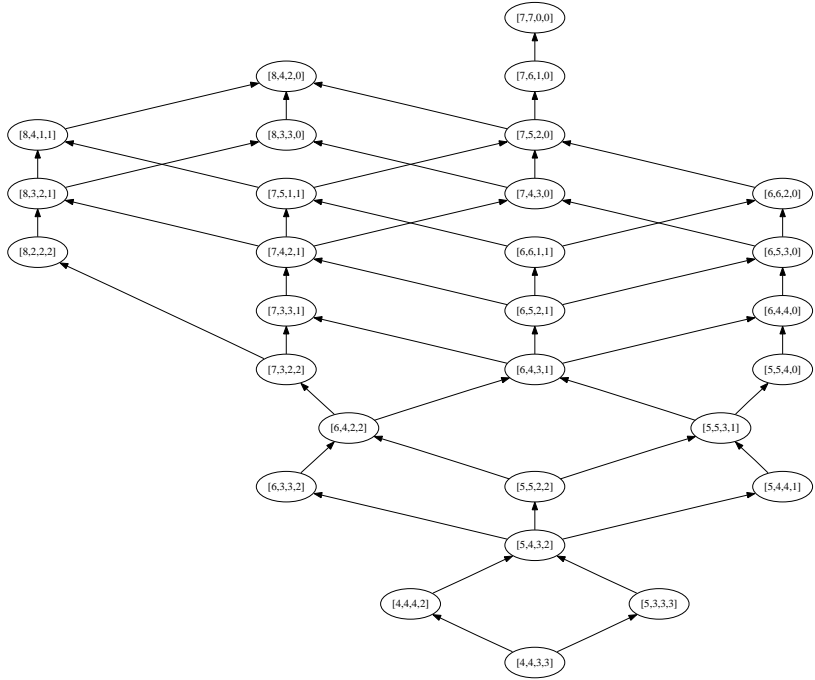}
    \caption{Hasse diagram of $B(p)$.  Arrows point from smaller to bigger
      elements.}\label{poset}
  \end{figure}
  One checks by computer that $p$ is translation invariant.  Since it is
  symmetric and homogeneous but lacks a maximum (Figure~\ref{poset}), it breaks
  Haldane's conjecture.  By Proposition~\ref{nonHaldane} it is minimal with
  respect to arity and homogeneous degree.
\end{proof}

\begin{rem*}
  The counterexample of Proposition~\ref{counterexample} is not minimal with
  respect to homogeneous degree.  For instance, the smallest pentavariate
  counterexamples have homogeneous degree $10$.
\end{rem*}

\begin{rem*}
  We might weaken Haldane's conjecture by hoping $R_n^d$ has a basis of Haldane
  polynomials.  Computer evidence suggests $|L_n^d| = \dim R_n^d$.  Writing
  $L_n^d = \{l_1, \ldots, l_k\}$, we could then obtain a special basis $\{p_1,
  \ldots, p_k\}$ of $R_n^d$ satisfying $B(p_i) \cap L_n^d = \{l_i\}$.  Perhaps
  it would be a Haldane basis or could be used to construct one.
\end{rem*}

\section{Proof of Theorem~\ref{ker_avg}}
\label{bigproof}

We factor $\rho$ into two maps which are easier to study:
\[
\xymatrix{X \ar[r]_\tau \ar@(ur,ul)[rr]^\rho & Y \ar[r]_\pi & \trans{Z}}
\]
Let $Y = \F[y_1, \ldots, y_n]$, and define $\tau, \pi$ by
\begin{align*}
  \tau(x_i) &= \frac{1}{n}\sum_{j=0}^{n-1} (n-1-j)y_{i+j} & \pi(y_i) &=
  z_i-z_{i+1}
\end{align*}
where index addition is modulo $n$.  Then
\[ \pi \tau (x_i) = \frac{1}{n}\biggl((n-1)z_i-\sum_{j \neq i}z_j\biggr) =
z_i-z_\avg
\]
showing $\pi \tau = \rho$.  It suffices to show $\tau^{-1}(\ker \pi) =
(x_\avg)$.  Since $ \tau(x_\avg) \propto y_\avg$, this follows from Lemmas
\ref{lem-isom} and \ref{lem-chain_isom} below.

\begin{lem} \label{lem-isom} $\tau$ is an isomorphism.
\end{lem}

\begin{proof}
  Let $\hat{\tau}\colon \F x_1 + \cdots + \F x_n \to \F y_1 + \cdots + \F y_n$
  be the linear map which extends to $\tau$. Then the matrix $M$ of
  $\hat{\tau}$ with respect to the evident bases is the $n \times n$ circulant
  matrix with first column vector
  \[v = \frac{1}{n}(n-1,n-2, \ldots, 0). \] Then $M^\top$ is the circulant
  matrix with first row $v$.  Since $\mathrm{char}{\F}=0$, the entries of $v$
  form a strictly decreasing sequence of nonnegative reals.  Therefore $M^\top$
  is nonsingular by Theorem~3 of Geller, Kra, Popescu, and Simanca
  \cite{circulant}.  Hence $\hat{\tau}$ is an isomorphism.  Therefore $\tau$ is
  an isomorphism by Observation~\ref{obs-matrix}.
\end{proof}

\begin{obs} \label{obs-matrix} Suppose $f: \F[a_1, \ldots, a_n] \to \F[b_1,
  \ldots, b_n]$ is an algebra homomorphism between polynomial rings which
  restricts to a linear map
  \[
  \hat{f}:\F a_1+\cdots +\F a_n \to \F b_1+\cdots +\F b_n
  \]
  If $\hat{f}$ is an isomorphism, then so is $f$.
\end{obs}

\begin{proof}
  The universal property of polynomial rings.
\end{proof}

\begin{lem}\label{lem-chain_isom}
  $\ker \pi = (y_\avg)$.
\end{lem}

\begin{proof}
  Let $\alpha = (\alpha_1,\alpha_2,\alpha_3)$ be the chain map
  \[\xymatrix{
    0 \ar[r] & (y_\avg) \ar[r] & Y \ar[r]^\pi & \trans{Z} \ar[r] & 0 \\
    0 \ar[r] & (y_1) \ar[u]^{\alpha_1} \ar[r] & Y \ar[u]^{\alpha_2}
    \ar[r]_{\pi'} & Y' \ar[u]_{\alpha_3} \ar[r] & 0 }\]
  where $Y' = \F[y_2, \ldots, y_n]$, and $\pi', \alpha_3, \alpha_2, \alpha_1$
  are the algebra homomorphisms such that $\pi'$ kills $y_1$ and fixes the
  other variables, $\alpha_3$ and $\pi$ agree, $\alpha_2$ sends $y_1$ to
  $y_\avg$ and fixes the other variables, $\alpha_1$ and $\alpha_2$ agree, and
  the unlabelled nonzero maps are inclusions of ideals.  We want the top
  sequence to be exact.  Since the bottom sequence is exact, it suffices to
  check $\alpha$ is a chain isomorphism.

  Since $\alpha$ is a chain map, it suffices to show each component is an
  isomorphism.  By Observation~\ref{obs-matrix}, $\alpha_2$ is an isomorphism.
  Then so is $\alpha_1$.  For $\alpha_3$, let $\beta\colon Y' \to Y'$ be the
  algebra homomorphism given by $\beta(y_i)= y_i+\cdots + y_n$ for $2 \leq i
  \leq n$.  Again by Observation~\ref{obs-matrix}, $\beta$ is an isomorphism.
  Then it suffices to show $\gamma = \alpha_3 \beta$ is an isomorphism.

  Note $\gamma\colon Y' \to \trans{Z}$ and $\gamma(y_i) = z_i-z_1$ for $2 \leq
  i \leq n$.  Since
  \[ p(z_1, \ldots, z_n)=p(0,z_2-z_1, \ldots, z_n-z_1)=\gamma(p(0,y_2, \ldots,
  y_n))
  \]
  for any $p(z_1, \ldots, z_n) \in \trans{Z}$, the homomorphism $\gamma$ is
  surjective.  If
  \begin{align*}
    0 = \gamma(q(y_2, \ldots, y_n))= q(z_2-z_1, \ldots, z_n-z_1)
  \end{align*}
  then $0 = \pi'(q(y_2-y_1, \ldots, y_n-y_1)) = q(y_2, \ldots, y_n)$, showing
  $\gamma$ is injective.  Thus $\gamma$ is an isomorphism.
\end{proof}

\section*{Acknowledgements}
It is a pleasure to thank Tobias Hagge, Thomas Howard, Joules Nahas, and
Zhenghan Wang for useful discussions, and Zhenghan Wang for suggesting the two
problems of this paper.

\bibliography{Haldaneandsympolysbib}
\bibliographystyle{hplain}
\nocite{*}
\end{document}